\documentclass[11pt]{article}
\usepackage{graphicx}
\usepackage{amsmath,amsfonts,amssymb,amsthm}
\title{On the cells in a stationary Poisson hyperplane mosaic}
\author{Matthias Reitzner and Rolf Schneider}
\date{}
\newcommand{\Sd}{{\mathbb S}^{d-1}}

\newcommand{\Kd}{{\mathcal K}^d}

\newcommand{\bP}{{\mathbb P}}

\newcommand{\Rd}{{\mathbb R}^d}
\newcommand{\N}{{\mathbb N}}

\newcommand{\Ha}{\mathcal{H}}

\newcommand{\B}{\mathcal{B}}
\newcommand{\D}{{\rm d}}

\newcommand{\bE}{{\mathbb E}\,}

\newcommand{\bQ}{{\mathbb Q}}

\newtheorem{theorem}{Theorem}
\newtheorem{lemma}{Lemma}

\begin{document}
\maketitle

\begin{abstract}
Let $X$ be the mosaic generated by a stationary Poisson hyperplane process $\widehat X$ in $\Rd$. Under some mild conditions on the spherical directional distribution of $\widehat X$ (which are satisfied, for example, if the process is isotropic), we show that with probability one the set of cells ($d$-polytopes) of $X$ has the following properties. The translates of the cells are dense in the space of convex bodies. Every combinatorial type of simple $d$-polytopes is realized infinitely often by the cells of $X$. A further result concerns the distribution of the typical cell. \\[2mm]
\end{abstract}

\section{Introduction}\label{sec1}

Consider a system ${\mathcal H}$ of hyperplanes in Euclidean space $\Rd$ which is locally finite, that is, every bounded subset of $\Rd$ is intersected by only finitely many hyperplanes from $\Ha$. The components of $\Rd\setminus\bigcup_{H\in\Ha} H$ are open convex polyhedra. Their closures are called {\em cells}, and the set of all cells is the {\em mosaic} induced by $\Ha$. A random process $\widehat X$ of hyperplanes in $\Rd$ induces a random mosaic $X$. If $\widehat X$ is stationary (its distribution is invariant under translations), then also the mosaic $X$ is stationary. Under some precautions, its cells are bounded and thus convex polytopes. The shapes of the cells in such a mosaic depend, of course, on the directions of the hyperplanes in $\widehat X$. For example, if $\widehat X$ is a parallel process, which means that its hyperplanes have only $d$ different directions (which are linearly independent), then all cells are parallelepipeds. On the other hand, if the hyperplane process $\widehat X$ is 
isotropic (its distribution is also invariant under rotations), then an inspection of some simulated examples will lead to the impression that the shapes can be quite varying and general. The purpose of this note is to substantiate this impression in the case of Poisson hyperplane processes. Due to the strong independence properties of Poisson processes, the variability of the shapes of the induced cells shows some extreme and perhaps unexpected features.

We make the following assumptions (for explanations, see the next section). We are given a stationary Poisson hyperplane process $\widehat X$ in $\Rd$, with a locally finite intensity measure $\widehat\Theta\not\equiv 0$. Let $\widehat\varphi$ be its spherical directional distribution. This is an even Borel measure on the unit sphere $\Sd$, which controls the directions of the hyperplanes appearing in $\widehat X$.\\[1mm]
{\bf Assumption (A):} The support of the spherical directional distribution $\widehat\varphi$ is the whole unit sphere $\Sd$. \\[1mm]
{\bf Assumption (B):} The spherical directional distribution $\widehat\varphi$ assigns measure zero to each great subsphere of $\Sd$.

Both assumptions are satisfied, for example, if $\widehat X$ is isotropic; in that case, the spherical directional distribution is the normalized spherical Lebesgue measure.

As mentioned, the random mosaic induced by $\widehat X$ is denoted by $X$. By $\Kd$ we denote the space of convex bodies (nonempty, compact, convex subsets) of $\Rd$. Its topology is induced by the Hausdorff metric $\delta$. The polytopes form a dense subset of $\Kd$.

\begin{theorem}\label{Thm1}
If assumptions $\rm (A)$ and $\rm (B)$ are satisfied, then with probability one, the set of all translates of the cells of $X$ is dense in $\Kd$.
\end{theorem}

\begin{theorem}\label{Thm2}
If assumptions $\rm (A)$ and $\rm (B)$ are satisfied, then with probability one, for every simple $d$-polytope $P$ there are infinitely many cells of $X$ that are combinatorially isomorphic to $P$.
\end{theorem}

The typical cell of the mosaic $X$ is a certain random polytope; see Section \ref{sec2}. Its distribution is a Borel measure on the space $\Kd_0$ of convex bodies in $\Rd$ with center at the origin; here the `center' refers to some continuous, translation covariant center function on the space of convex bodies, for example, the center of the circumball.

\begin{theorem}\label{Thm3}
If assumption $\rm (A)$ is satisfied, then  the support of the distribution of the typical cell is the whole space $\Kd_0$.
\end{theorem}

\section{Notation and explanations}\label{sec2}

For a convex body $K\in\Kd$ and a number $\eta \ge 0$, we denote by $K_\eta:= K+\eta B^d$ (where $B^d$ denotes the unit ball of $\Rd$) the parallel body of $K$ at distance $\eta$. The Hausdorff distance of convex bodies $K,L\in\Kd$ is defined by 
$$ \delta(K,L)= \min\{\eta\ge 0: K\subset L_\eta,\; L\subset K_\eta\}.$$
Then $\delta$ is a metric on $\Kd$. Topological notions for $\Kd$ refer to the topology induced by this metric. In particular, $\B(\Kd)$ denotes the $\sigma$-algebra of Borel sets in $\Kd$.

By $\Ha^d$ we denote the space of hyperplanes in $\Rd$, with its usual topology. The $\sigma$-algebra of Borel sets in $\Ha^d$ is denoted by $\B(\Ha^d)$. For a subset $M\subset\Rd$ we define
$$ \Ha(M):=\{H\in\Ha^d: H\cap M\not=\emptyset\}.$$

In the following, notation concerning stochastic geometry is as in \cite{SW08}, in particular Section 10.3, to which we also refer for more detailed information. As already mentioned, $\widehat X$ is assumed to be a stationary Poisson hyperplane process in $\Rd$, thus, a Poisson point process in the space $\Ha^d$, whose distribution is invariant under translations. Since we consider only simple point processes, it is convenient to identify a simple counting measure with its support. Thus, in the following, the realizations of $\widehat X$ are considered as locally finite systems of hyperplanes. The intensity measure $\widehat\Theta$ of $\widehat X$ is defined by
$$ \widehat\Theta(A)= \bE |\widehat X\cap A| \quad\mbox{for }A\in\B(\Ha^d),$$
where $|\cdot|$ denotes the number of elements (we denote expectations by $\bE$, and the probability by $\bP$). We assume that $\widehat\Theta$ is locally finite and not identically zero. Due to the stationarity assumption, the measure $\widehat\Theta$ has a decomposition: there are a number $\widehat\gamma>0$, the {\em intensity} of $\widehat X$, and an even Borel probability measure $\widehat\varphi$ on the unit sphere $\Sd$ of $\Rd$ such that
\begin{equation}\label{2.1}
\widehat\Theta(A) = \widehat\gamma \int_{\Sd} \int_{-\infty}^\infty {\bf 1}_A(u^\perp+\tau u)\,\D\tau\,\widehat\varphi(\D u)
\end{equation}
for $A\in\B(\Ha^d)$ (see \cite[Theorem 4.4.2 and (4.33)]{SW08}). Here ${\bf 1}_A$ denotes the indicator function  of $A$, and $u^\perp$ is the hyperplane through $0$ orthogonal to the unit vector $u$. The measure $\widehat \varphi$ is called the {\em spherical directional distribution} of $\widehat X$. 

The Poisson assumption says that
$$ \bP \{|\widehat X\cap A|=n \} = e^{-\widehat\Theta(A)} \frac{\widehat\Theta(A)^n}{n!}\quad\mbox{for }n\in\N_0,$$
for all $A\in\B(\Ha^d)$ with $\widehat\Theta(A)<\infty$. If $A_1,A_2, \ldots\in \B(\Ha^d)$ are pairwise disjoint, then the restricted processes $\widehat X\cap A_1,\widehat X\cap A_2,\dots$ are stochastically independent (see, e.g., \cite[Theorem 3.2.2]{SW08}). This fact is crucial for the results of the present note.

The mosaic $X$ induced by $\widehat X$ is usually considered as a particle process (see, e.g., \cite[Chapter 10]{SW08}); with our convention, the realizations of $X$ are certain sets of polytopes in $\Rd$. The intensity measure $\Theta^{(d)}$ of the mosaic $X$ is defined by
$$ \Theta^{(d)}(A) = \bE|X\cap A|\quad\mbox{for }A\in \B(\Kd).$$
By stationarity, it, too, has a decomposition. For this, we choose any continuous function $c:\Kd\to\Rd$ which is translation covariant and satisfies $c(K)\in K$, for example, the center of the circumball or the Steiner point. This function $c$ is called the {\em center function}, and we denote by $\Kd_0$ the set of all convex bodies $K\in\Kd$ with $c(K)=0$. With this choice, there exist a number $\gamma^{(d)}>0$, the {\em intensity} of $X$, and a Borel probability measure $\bQ^{(d)}$ on $\Kd_0$ such that
\begin{equation}\label{2.2} 
\Theta^{(d)}(A) = \gamma^{(d)} \int_{\Kd_0} \int_{\Rd} {\bf 1}_A(C+x)\,\lambda(\D x)\,\bQ^{(d)}(\D C)
\end{equation}
for $A\in\B(\Kd)$, where $\lambda$ denotes Lebesgue measure on $\Rd$ (see \cite{SW08}, Theorem 4.1.1 and (4.3)). Clearly, the measure $\bQ^{(d)}$ is concentrated on the set of polytopes in $\Kd_0$ (which is a Borel set). 

The {\em typical cell} of the random mosaic $X$, denoted by $Z$, is now defined as a random polytope with distribution $\bQ^{(d)}$. A more intuitive interpretation of the typical cell is obtained as follows. Let $W\in\Kd$ be a convex body with interior points. Then, for $A\in\B(\Kd_0)$,
$$ \bP\left\{Z\in A\right\}=\lim_{r\to\infty} \frac{\bE \sum_{C\in X,\, C\subset rW} {\bf 1}_A(C-c(C))}
{\bE \sum_{C\in X,\,C\subset rW} 1}.$$
This can be deduced from \cite[Theorem 4.1.3]{SW08}.

\section{Proofs of the theorems}\label{sec3}

We shall need the following generalization of the Borel--Cantelli lemma. 

\begin{lemma}\label{Lem3.1}
Let $E_1,E_2,\dots$ be a sequence of events (on some probability space) with $\sum_{j=1}^\infty \bP(E_j) =\infty$ and 
$$ \liminf_{n\to\infty}\frac{\sum^n_{i,j=1,\, i\neq j} [\bP(E_i\cap E_j)-\bP(E_i)\bP(E_j)]}{(\sum_{j=1}^n \bP(E_j))^2}= 0.$$
Then $\bP(\limsup_{j\to\infty} E_j)=1$.
\end{lemma}

This is a slight reformulation, convenient for our purposes, of a result by Erd\"os and R\'{e}nyi \cite{ER59}; see also \cite[p. 327, Hilfssatz C]{Ren66}. In fact, with $\lambda_n= \sum_{j=1}^n \bP(E_j)$, the identity
$$ \frac{1}{\lambda_n^2} \sum_{i,j=1,\,i\not= j}^n [\bP(E_i\cap E_j)-\bP(E_i)\bP(E_j)] = \frac{\sum_{i,j=1}^n \bP(E_i\cap E_j)}{\lambda_n^2} -1 +\frac{\sum_{j=1}^n \bP(E_j)^2 -\lambda_n}{\lambda_n^2}$$
holds, so that our assumptions imply the assumptions of Erd\"os and R\'{e}nyi.

We prepare the proofs of Theorems \ref{Thm1}, \ref{Thm2}, \ref{Thm3} by some geometric considerations and corresponding lemmas.

For a polytope $Q$, we denote by ${\rm vert}\,Q$ the set of vertices of $Q$. In the following two lemmas, $P\subset\Rd$ is a given convex polytope with interior points. Let $F_1,\dots,F_m$ be the facets of $P$. By $B(v,\varepsilon)$ we denote the closed ball with center $v\in \Rd$ and radius $\varepsilon\ge 0$. For $j\in\{1,\dots,m\}$, we define
$$ A_j(P,\varepsilon) := \bigcap_{v\in{\rm vert}\,F_j} \Ha(B(v,\varepsilon)).
$$
Thus, a hyperplane belongs to $A_j(P,\varepsilon)$ if and only if it has distance at most $\varepsilon$ from each vertex of the facet $F_j$. Each hyperplane of $A_j(P,\varepsilon)$ is said to be {\em $\varepsilon$-close} to $F_j$. 

For every neighborhood ${\mathcal N}$ (in $\Ha^d$) of the affine hull of the facet $F_j$ we have $A_j(P,\varepsilon)\subset{\mathcal N}$ for all sufficiently small $\varepsilon$. Therefore, for sufficiently small $\varepsilon>0$, the following holds. The sets $A_1(P,\varepsilon), \dots, A_m(P,\varepsilon)$ are pairwise disjoint. If $H_j\in A_j(P,\varepsilon)$ for $j=1,\dots,m$, then each hyperplane $H_j$ determines a closed halfspace that contains the vertices of $P$ that are not vertices of $F_j$, and the intersection of these halfspaces is a convex polytope. Such a polytope is said to be {\em $\varepsilon$-close} to $P$. 

\begin{lemma}\label{Lem3.2}
To every $\eta>0$, there exists $\varepsilon>0$ such that $A_1(P,\varepsilon), \dots, A_m(P,\varepsilon)$ are pairwise disjoint and that any hyperplanes $H_j\in A_j(P,\varepsilon)$, $j=1,\dots,m$, are the facet hyperplanes of a polytope $Q$ that satisfies $\delta(Q,P)\le \eta$.
\end{lemma}

\begin{proof}
The remaining assertion follows easily from the following facts. If to each vertex $x$ of $Q$ there is a vertex $v$ of $P$ with $x\in B(v,\eta)$, then $Q\subset P_\eta$ (and similarly, with $P$ and $Q$ interchanged). If the sequence $(H_{ir})_{r\in\N}$ of hyperplanes converges to the hyperplane $H_i$, for $i=1,\dots,d$, and if $H_1,\dots,H_d$ have linearly independent normal vectors, then for almost all $r$, also $H_{1r},\dots,H_{dr}$ have independent normal vectors, and their intersection point converges to the intersection point of $H_1,\dots,H_d$.
\end{proof}

\begin{lemma}\label{Lem3.3}
Let $\varepsilon_0>0$. With probability one, there are infinitely many cells in $X$ such that for each of these cells a translate is $\varepsilon_0$-close to $P$.
\end{lemma}

\begin{proof}

We choose a number $\eta>0$ and then, according to Lemma \ref{Lem3.2}, a number $0<\varepsilon\le \varepsilon_0$ such that that $A_1(P,\varepsilon), \dots, A_m(P,\varepsilon)$ are pairwise disjoint and that $\delta(Q,P)\le \eta$ if the polytope $Q$ is $\varepsilon$-close to $P$. In particular, all polytopes that are $\varepsilon$-close to $P$ are contained in the parallel body $P_\eta$.

We extend the definition of $A_j(P,\varepsilon)$ to the translates of P. For $t\in\Rd$, let 
$$ A_j(P+t,\varepsilon) := \bigcap_{v\in{\rm vert}(F_j+t)} \Ha(B(v,\varepsilon)).$$
Further, set
$$ A(P+t,\varepsilon) := \bigcup_{j=1}^m A_j(P+t,\varepsilon).$$

\vspace{2mm}

\noindent{\bf Definition.} 
Let $t_1,t_2\in\Rd$. The polytopes $P+t_1$ and $P+t_2$ are {\em $\varepsilon$-disentangled} if 
$$ A(P+t_1,\varepsilon)\cap \Ha(P_{\eta}+t_2) = \emptyset \quad \mbox{and} \quad A(P+t_2,\varepsilon)\cap \Ha(P_{\eta}+t_1)=\emptyset.$$ 

\vspace{2mm}

For $\alpha\ge 0$, let $\Ha(\alpha)$ be the set of all hyperplanes through $0$ that are parallel to some hyperplane in $A(P,\alpha)$. Let
$$ M(\alpha) := \Rd \setminus \bigcup_{H\in\Ha(\alpha)} H.$$
The set $M(\alpha)$ is open and is a cone, that is, if $x\in M(\alpha)$, then $\lambda x\in M(\alpha)$ for all $\lambda>0$. Trivially, there exists a line $G$ through $0$ satisfying
$$ G\setminus\{0\}\subset M(0).$$
It follows easily that we can decrease $\varepsilon>0$, if necessary, such that
$$ G\setminus\{0\}\subset M(\varepsilon).$$
Let $t\in G$, $t\not=0$. Since $M(\varepsilon)$ is a cone and since also $-t\in G$, we can choose $\mu>0$ so large that the polytopes $P$ and $P+\mu t$ are $\varepsilon$-disentangled (note that the distance of a hyperplane in $A(P,\varepsilon)$ from the parallel hyperplane in $\Ha(\varepsilon)$ is bounded by some constant depending only on $P$ and $\varepsilon$). After this choice, we write $t$ instead of $\mu t$, so that now $P$ and $P+t$ are $\varepsilon$-disentangled. As is clear from the definitions, any two polytopes $P+t_1$ and $P+t_1+\lambda t$ with $\lambda\ge 1$ are $\varepsilon$-disentangled. 

Let
$$ C(P,\varepsilon):= \Ha(P_{\eta})\setminus A(P,\varepsilon).$$
Thus, $C(P,\varepsilon)$ is the set of hyperplanes that meet the parallel body $P_\eta$, but are not $\varepsilon$-close to some facet of $P$.

\vspace{2mm}

\noindent{\bf Definition.}  $E(P,\varepsilon)$ is the event that
$$ |\widehat X\cap A_j(P,\varepsilon)|=1\mbox{ for }j=1,\dots,m, \; |\widehat X\cap C(P,\varepsilon)|=0.$$

Suppose that the event $E(P,\varepsilon)$ occurs.  Then for each $j\in \{1,\dots,m\}$, there is precisely one hyperplane $H_j$ of $\widehat X$ in the set $A_j(P,\varepsilon)$. The hyperplanes $H_1,\dots,H_m$ are the facet hyperplanes of a polytope $Q$, which is $\varepsilon$-close to $P$. In the event $E(P,\varepsilon)$, there is no hyperplane of $\widehat X$ in the set $C(P,\varepsilon)$, hence no hyperplane of $\widehat X$ different from $H_1,\dots,H_m$ meets $P_{\eta}$ (which contains $Q$). Therefore, $Q$ is a cell of the mosaic $X$. Thus, if $E(P,\varepsilon)$ occurs, then some cell of $X$ is $\varepsilon$-close to $P$.

Our choice of $\varepsilon$ implies, in particular, that the sets $A_1(P,\varepsilon),\dots,A_m(P,\varepsilon), C(P,\varepsilon)$ are pairwise disjoint, hence the restrictions 
$$\widehat X\cap A_1(P,\varepsilon),\dots,\widehat X\cap A_m(P,\varepsilon),\widehat X\cap C(P,\varepsilon)$$ 
are stochastically independent. It follows that
\begin{align*}
&\bP(E(P,\varepsilon))\\
&= \bP\{ |\widehat X\cap C(P,\varepsilon))|=0\} \prod_{j=1}^m \bP\{|\widehat X\cap A_j(P,\varepsilon)|=1\}\\
&= e^{-\widehat\Theta(C(P,\varepsilon))}\prod_{j=1}^m e^{-\widehat \Theta(A_j(P,\varepsilon))}\widehat\Theta(A_j(P,\varepsilon)).
\end{align*}
From (\ref{2.1}) and Assumption (A), we immediately obtain that $\widehat \Theta(A_j(P,\varepsilon))>0$ for all $j$.
We have obtained  that
\begin{equation}\label{3.0} 
p(\varepsilon) := \bP(E(P,\varepsilon)) >0.
\end{equation}

Above, we have found a vector $t\in\Rd$ such that $P+t_1$ and $P+t_1+\lambda t$ are $\varepsilon$-disentangled for any $t_1\in\Rd$ and any $\lambda\ge 1$. With numbers $\lambda_1,\lambda_2,\ldots\ge 1$ to be determined later, we now define recursively
$$ P_1 = P, \quad P_{n+1} = P_n +\lambda_n t \quad\mbox{for }n\in\N.$$
Any two polytopes of the set $\{P_n:n\in\N\}$ are $\varepsilon$-disentangled.

\vspace{2mm}

\noindent{\bf Definition.} For $i\in\N$, let $E_i$ denote the event that
$$ |\widehat X\cap A_j(P_i,\varepsilon)|=1\mbox{ for }j=1,\dots,m,\; |\widehat X\cap C(P_i,\varepsilon)|=0.$$

Since the events $E_i,E_k$ with $i\not= k$ are not independent, we modify this. Let
$$ \Ha_i=\Ha((P_i)_{\eta}), $$
the set of hyperplanes meeting $(P_i)_\eta$.

\vspace{2mm}

\noindent{\bf Definition.} For $i\not= k$, let $E_{ik}$ denote the event that 
$$ |\widehat X\cap A_j(P_i,\varepsilon)|=1\mbox{ for }j=1,\dots,m,\; |\widehat X\cap C(P_i,\varepsilon)\cap \Ha_k^c|=0,$$
where $\Ha_k^c:=\Ha\setminus \Ha_k$, and let $H_{ik}$ be the event that
$$|\widehat X\cap \Ha_i\cap \Ha_k|=0.$$

If the event $E_{ik}$ occurs, then the mosaic induced by $\widehat X \setminus(\Ha_i\cap \Ha_k)$ has a cell that is $\varepsilon$-close to $P_i$. The event $H_{ik}$ ensures that $\widehat X$ has no hyperplanes in $\Ha_i\cap\Ha_k$. Thus, in this case some cell of the mosaic $X$ is $\varepsilon$-close to $P_i$. Clearly we have $E_{ik}\cap H_{ik}\subset E_i\subset E_{ik}$.

Let $i<k$ be given. Since $E_i\subset E_{ik}$, 
\begin{equation}\label{3.1}
\bP(E_i)\le \bP(E_{ik}),\quad \bP(E_k)\le \bP(E_{ki})
\end{equation}
and
\begin{equation}\label{3.2}
\bP(E_i\cap E_k)\le \bP(E_{ik}\cap E_{ki}).
\end{equation}
Since the sets $\Ha_i\cap \Ha_k^c$ and $\Ha_k\cap \Ha_i^c$ are disjoint, the events $E_{ik}$ and $E_{ki}$ are independent, hence 
\begin{equation}\label{3.3}
\bP(E_{ik}\cap E_{ki}) = \bP(E_{ik})\bP(E_{ki}).
\end{equation}
Because $E_{ik}\cap H_{ik} \subset E_i$ and the events $E_{ik}$ and $H_{ik}$ are independent, we have
\begin{equation}\label{3.4}
\bP(E_{ik})\bP(H_{ik}) \le \bP(E_i),\quad \bP(E_{ki})\bP(H_{ik}) \le \bP(E_k).
\end{equation}
Further, $E_{ik}\cap E_{ki}\cap H_{ik}\subset E_i\cap E_k$, and the events $E_{ik}\cap E_{ki}$ and $H_{ik}$ are independent, hence
\begin{equation}\label{3.5}
\bP(E_{ik}\cap E_{ki}) \bP(H_{ik})\le \bP(E_i\cap E_k).
\end{equation}
From (\ref{3.2}), (\ref{3.4}), (\ref{3.3}) we get
$$ \bP(E_i\cap E_k)-\bP(E_i)\bP(E_k) \le \bP(E_{ik})\bP(E_{ki})[1-\bP(H_{ik})^2] \le 1-\bP(H_{ik})^2,$$
and from (\ref{3.5}), (\ref{3.1}), (\ref{3.3}),
$$ -[\bP(E_i\cap E_k)-\bP(E_i)\bP(E_k)] \le \bP(E_{ik})\bP(E_{ki})[1-\bP(H_{ik})] \le 1-\bP(H_{ik}).$$

By choosing $\lambda_i$ sufficiently large, we can make $\bP(H_{ik})$ arbitrarily close to $1$. In fact, we have
$$ \bP (H_{ik}) = e^{-\widehat\Theta(\Ha_i\cap\Ha_k)}$$
and, by (\ref{2.1}), 
$$ \widehat \Theta(\Ha_i\cap\Ha_k) =\widehat\gamma \int_{\Sd} \int_{-\infty}^\infty {\bf 1}_{\Ha_i\cap\Ha_k}(u^\perp+\tau u)\,\D\tau\,\widehat\varphi(\D u).$$
The inner integral is bounded by the diameter of $P_{\eta}$. The outer integral extends in effect only over a neighborhood of the great subsphere $t^\perp\cap\Sd$, and for $\lambda_i\to\infty$, these neighborhoods shrink to $t^\perp\cap \Sd$. By our Assumption (B), $\widehat\varphi(t^\perp\cap \Sd)=0$. This gives $\widehat\Theta(\Ha_i\cap\Ha_k)\to 0$ for $\lambda_i\to\infty$.

Therefore, we can assume that
$$ |\bP(E_i\cap E_k)-\bP(E_i)\bP(E_k)| \le 1/i^2$$
for all $i$ and all $k>i$. This gives
$$ \sum_{i\not=k}^n|\bP(E_i\cap E_k)-\bP(E_i)\bP(E_k)|\le cn$$
with a constant $c$ independent of $n$. Since 
$$ \sum_{i=1}^n\bP(E_i)>p(\varepsilon)n$$
with $p(\varepsilon)>0$ (which follows from (\ref{3.0}) and stationarity), Lemma \ref{Lem3.1} gives $\bP(\limsup_{n\to\infty} E_n)=1$. Thus, with probability one, infinitely many events $E_n$ occur. But if an event $E_n$ occurs, then there is a cell in $X$ (contained in $(P_n)_\eta$) such that a translate of it is $\varepsilon$-close to $P$. Since $\varepsilon\le\varepsilon_0$, this completes the proof of Lemma \ref{Lem3.3}.
\end{proof}

Now we are in a position to finish the proofs of our theorems.

\begin{proof}[Proof of Theorem \ref{Thm1}]
The set of all $d$-polytopes in $\Rd$ whose vertices have rational coordinates is countable and hence can be ordered as a sequence $Q_1,Q_2,\dots$ This sequence is dense in $\Kd$ with respect to the Hausdorff metric. 

We choose a sequence $(\eta_k)_{k\in\N}$ of positive numbers with $\lim_{k\to\infty}\eta_k=0$. By Lemma \ref{Lem3.2}, to each $k$ we can choose a number $\varepsilon_k>0$ such that $A_1(Q_k,\varepsilon_k), \dots, A_{m_k}(Q_k,\varepsilon_k)$ (where $m_k$ is the facet number of $Q_k$) are pairwise disjoint and that every polytope $C$ that is $\varepsilon_k$-close to $Q_k$ satisfies $\delta(Q_k,C)\le \eta_k$.

Let $k\in\N$. By Lemma \ref{Lem3.3}, with probability one there is a cell $C_k$ in $X$ such that a translate $C_k+t_k$ is $\varepsilon_k$-close to $Q_k$, in particular, $\delta(Q_k,C_k+t_k)\le \eta_k$. Since the intersection of countably many events of probability one still has probability one, then with probability one for each $k\in\N$ there is a cell $C_k$ in $X$ with $\delta(Q_k,C_k+t_k)\le\eta_k$ for suitable $t_k$. 

Consequently, the following holds with probability one. Let $K\in\Kd$. There is a subsequence $(Q_{k_r})_{r\in\N}$ of the dense sequence $(Q_k)_{k\in\N}$ that converges to $K$. To each $r$, there is a translated cell $C_{k_r}+t_{k_r}$ with $C_{k_r}\in X$ and $t_{k_r}\in\Rd$ such that $\delta(Q_{k_r},C_{k_r}+t_{k_r})\le\eta_{k_r}$. Then also the sequence $(C_{k_r}+t_{k_r})_{r\in\N}$ converges to $K$. Since $K\in\Kd$ was arbitrary, this shows that the translates of the cells of $X$ are dense in $\Kd$.
\end{proof}

\begin{proof}[Proof of Theorem \ref{Thm2}]
There are only countably many combinatorial isomorphism types of simple $d$-polytopes. Therefore, we can choose a sequence $(Q_k)_{k\in\N}$ of simple $d$-polytopes which represent all these combinatorial types. 

Let $k\in\N$. Since $Q_k$ is simple, there is a number $\varepsilon>0$ such that every polytope that is $\varepsilon$-close to $Q_k$ must be combinatorially isomorphic to $Q_k$. By Lemma \ref{Lem3.3}, there are infinitely many cells in $X$ such that for each of these cells a translate is $\varepsilon$-close to $Q_k$, hence the cell is combinatorially isomorphic to $Q_k$. Since this holds for each $k$ with probability one, it holds with probability one simultaneously for all $k$.
\end{proof}

\begin{proof}[Proof of Theorem \ref{Thm3}] Let $\emptyset\not= B\subset\Kd_0$ be an open set. There is a polytope $P\in B$, and we can choose a number $\eta>0$ such that the $\eta$-neighborhood of $P$ with respect to the Hausdorff metric is contained in $B$. 

The center function $c$ is continuous at $P$, hence there exists a number $\varepsilon_1>0$ such that $\|c(Q)\|=\|c(Q)-c(P)\| \le\eta/2$ if $\delta(Q,P)< \varepsilon_1$. By Lemma \ref{Lem3.2}, there exists a number $\varepsilon_2>0$ such that every polytope $Q$ that is $\varepsilon_2$-close to $P$ satisfies $\delta(Q,P) \le \min\{\varepsilon_1,\eta/2\}$. Let $\varepsilon := \min\{\varepsilon_1,\varepsilon_2\}$.

As shown in the proof of Lemma \ref{Lem3.3}, with probability $p(\varepsilon)>0$ (see (\ref{3.0})) the mosaic $X$ contains a cell $C$ that is $\varepsilon$-close to $P$ (the proof uses only Assumption (A)). This cell satisfies $\delta(C,P) \le\eta/2$ and $\delta(C,P)\le \varepsilon_1$, and the latter gives $\|c(C)\|\le \eta/2$. Then we have $C-c(C)\in\Kd_0$ and $\delta(C-c(C),P) \le \delta(C,P) + \|c(C)\| \le \eta$ and thus $C-c(C)\in B$.

Define
$$ A:= \{K\in\Kd: c(K)\in P_\eta\mbox{ and }K-c(K)\in B\}.$$

On one hand, we have 
$$\Theta^{(d)}(A) = \bE|X\cap A| \ge p(\varepsilon)>0,$$
since with probability $p(\varepsilon)$, there is a cell $C$ of $X$ that is $\varepsilon$-close to $P$ and thus satisfies $C\subset P_\eta$ and hence $c(C)\in P_\eta$, and moreover $C-c(C)\in B$. 

On the other hand, by (\ref{2.2}),
\begin{align*}
\Theta^{(d)}(A) &= \gamma^{(d)} \int_{\Kd_0} \int_{\Rd} {\bf 1}_A(C+x)\,\lambda(\D x)\,\bQ^{(d)}(d C)\\
&= \gamma^{(d)} \lambda(P_\eta) \int_{\Kd_0} {\bf 1}_B(C)\,\bQ^{(d)}(d C) = \gamma^{(d)} \lambda(P_\eta) \bQ^{(d)}(B).
\end{align*}
Both results together show that $\bQ^{(d)}(B)>0$.
\end{proof}

\section{Poisson--Voronoi mosaics}

We remark that results analogous to Lemma \ref{Lem3.3} and hence to Theorems \ref{Thm1} and \ref{Thm2} hold also for stationary Poisson--Voronoi mosaics (for these, see \cite[Section 10.2]{SW08}, for example). They are easier to obtain, since there are no long-range dependences, so that the usual Borel--Cantelli lemma (for pairwise independent events) is sufficient and its generalization, Lemma \ref{Lem3.1}, is not needed. We sketch only the beginning of the proof.

Let $Y$ be a stationary Poisson point process in $\Rd$ with intensity $\gamma>0$. Let $P\subset \Rd$ be a polytope with $0$ as an interior point, and let $F_1,\dots,F_m$ be its facets. Let $p_1,\dots,p_m$ be the points obtained by reflecting the origin $p_0=0$ at each of the affine hulls of $F_1,\dots, F_m$. Then $P$ is a cell of the Voronoi diagram of the set $\{p_0,p_1,\dots,p_m\}$.

Let $\varepsilon>0$ be given. We can choose $\eta>0$ such that every polytope that is $\varepsilon$-close to $P$, in the sense defined in Section \ref{sec3}, is contained in $P_\eta$.
Next, we choose $\alpha>0$ so small that the Voronoi diagram of any set $\{q_0,\dots,q_m\}$ with $q_j\in B(p_j,\alpha)$, $j=0,\dots,m$, has a cell $Q$ that is $\varepsilon$-close to $P$. 
Further, there exists a number $\rho>0$ such that for any point $q\in\Rd\setminus \rho P_\eta$, the mid-hyperplane of $q$ and $0$ does not intersect $P_\eta$. Now we define
$$ C = \rho P_\eta \setminus \bigcup_{j=0}^m B(p_j,\alpha)$$
and let $E$ be the event that
$$ |Y\cap B(p_j,\eta)|= 1\mbox{ for } j=0,\dots,m,\enspace |Y\cap C|=0.$$
If the event $E$ occurs, then the Voronoi mosaic induced by $Y$ has a cell that is $\varepsilon$-close to $P$. This event has probability
$$ \bP(E)= e^{-\gamma \lambda(C)}\left[ e^{-\gamma\lambda(B(0,\eta))} \gamma\lambda(B(0,\eta))\right]^{m+1}>0.$$
The rest of the proof is left to the reader.

\noindent Authors' addresses:\\[2mm]
Matthias Reitzner\\Institut f\"ur Mathematik, Unversit\"at Osnabr\"uck\\
D-49076 Osnabr\"uck, Germany\\
E-mail: matthias.reitzner@uni-osnabrueck.de\\[2mm]
Rolf Schneider\\
Mathematisches Institut, Albert-Ludwigs-Universit{\"a}t\\
D-79104 Freiburg i. Br., Germany\\
E-mail: rolf.schneider@math.uni-freiburg.de

\end{document}